%% file: RFN_main.tex
\documentclass[hidelinks,onefignum,onetabnum]{siamart251216}

\usepackage{graphicx} 
\usepackage{amsmath}
\usepackage{amsfonts}
\usepackage{indentfirst}
\usepackage{multirow}
\usepackage{booktabs}
\usepackage{overpic}
\usepackage{makecell}
\usepackage{bbm}
\ifpdf
\hypersetup{
  pdftitle={Physics Informed Random Feature Neural Networks for Solving PDEs},
  pdfauthor={C. Chen, C. Liao, M. Zhong}
}
\fi

\headers{Physics-Informed RFN for Solving PDEs}{C. Chen, C. Liao, M. Zhong}

\title{Physics-Informed Random Feature Neural Networks for Solving PDEs\thanks{Submitted to the editors DATE.
\funding{Zhong is supported by NSF-CCF-AoF grant $\#2225507$. CL is supported by the start-up funding from University of Arkansas. }}}

\author{Chi-an Chen\thanks{Department of Applied Mathematics, Mathematics, Illinois Institute of Technology, Chicago, IL, $60616$
  (\email{cchen156@hawk.illinoistech.edu}, \url{https://scholar.google.com/citations?user=WdF92UgAAAAJ&hl=en}).}
\and Chunyang Liao\thanks{Department of Mathematical Sciences, University of Arkansas, Fayetteville, AR, $72701$
  (\email{cliao1@uark.edu}, \url{https://liaochunyang.github.io/}).}
\and Ming Zhong\thanks{Department of Mathematics, University of Houston, Houston, TX, $77204$
  (\email{mzhong3@central.uh.edu}, \url{https://mingjzhong.github.io/}).}}
\input{marco}

\begin{document}

\maketitle

\begin{abstract}
Machine learning-based partial differential equations (PDEs) solvers have attracted significant attention in recent years. Most progress in this area has been driven by deep neural networks such as physics-informed neural networks (PINNs) and kernel method (such as physics-informed Gaussian Processes).  
We introduce a physics-informed random feature method for countering part of the spectral bias which PINN-based solvers are facing for a certain class of PDEs. Random feature method was originally proposed to approximate large-scale kernel machines and can be viewed as a specialized randomized neural network.
Compared to other state-of-the-art PINN-based solvers which require a large number of collocation points, our proposed method reduces the computational complexity.  
In this paper, we develop a rigorous approximation error analysis and derive high-probability error bounds on the $H^1$ norm.
We provide extensive numerical tests for verifying our theoretical guarantees on error decay rates, as well as several comparison tests to showcase our claimed capability for combating spectral bias in these deep learning based methods.
\end{abstract}



\section{Introduction}\label{sec:intro}
Partial Differential Equations (PDEs) have been playing a crucial role in modeling natural phenomena.  Accurate numerical solutions are these PDE-driven models have long been a major research interest in the science and engineering communities.  Traditional numerical methods, including but not limited to Finite Difference, Finite Element and Spectral Methods, have been well developed for such a task for the past century~\cite{Eitan2012}.  Scientific Machine Learning (SciML), where machine learning techniques combined with scientific computing frameworks, have garnished increasing popularity in research communities in recent years~\cite{osti_1478744}.  Among these methods, Physics-Informed Neural Networks (PINNs) have emerged as one of the most influential approaches by embedding physical knowledge in terms of partial differential equations directly into the training objective~\cite{Raissi2019PINN}. Their mesh-free formulation, flexibility, and compatibility with automatic differentiation have enabled applications to a wide variety of forward and inverse PDE problems~\cite{PINN_review}. Despite these successes, PINNs remain challenging to train in practice due to the highly non-convex optimization over millions of network parameters, sensitivity to network architecture and hyperparameter choices, and the well-documented tendency to favor low-frequency solution components during training~\cite{WANG2022110768, Hao2026}.

Kernel methods provide a rigorous functional and analytic framework for function approximation and possess strong theoretical guarantees through reproducing kernel Hilbert spaces (RKHS)~\cite{BATLLE2025113488}. However, classical kernel methods suffer from poor scalability, since training requires manipulating dense kernel matrices whose computational cost grows cubically with the number of training points. Random feature methods provide an elegant compromise by approximating kernel functions using finite-dimensional randomized feature maps~\cite{Rahimi2007RFM}. 
Viewed from a deep learning perspective, random feature models are a special type of shallow neural networks whose hidden-layer parameters are randomly sampled and fixed, while only the output-layer coefficients are optimized. Unlike other randomized neural networks, such as Extreme Learning Machine (ELM) \cite{Huang2004ELM}, the hidden-layer weights and biases are sampled from desired probability distributions induced by the underlying kernel, thereby preserving the approximation properties of kernel methods while maintaining the computational efficiency of shallow neural networks. 

Motivated by these observations, we introduce Physics-Informed Random Feature Networks (PIRFNs) as a mathematically motivated addition to PINNs for solving PDEs. The central idea is to replace the trainable hidden layers in PINNs with kernel-induced random features while retaining the physics-informed loss formulation. As a result, only the output-layer coefficients are optimized, leading to a substantially simpler optimization problem together with rigorous approximation guarantees derived from theories based on RKHS. 

The objective of this work is to further develop physics-informed random feature networks into a more general and practical computational framework. From a theoretical perspective, we extend the approximation analysis to Sobolev spaces and establish new approximation results in the $H^1$ norm, providing stronger error estimates that are more relevant to Sobolev-space based Finite Element analysis. From an algorithmic perspective, we redesign the training procedure by focusing on the auto-differentiation workflow commonly used in PINNs together with the widely adopted Adam and L-BFGS optimizers. Consequently, the method retains the user friendly implementation pipeline of PINNs while preserving the computational advantages of random feature models.

Beyond the optimization framework, we also investigate how the design of random features influences approximation quality. Rather than employing a single random feature distribution for all variables, we introduce a product random feature construction that allows variables from different dimensions to be represented by different kernels and sampling distributions. This design is motivated by product kernels 
and provides additional flexibility for modeling evolutionary PDEs in which spatial and temporal variables possess distinct characteristic scales. Our numerical experiments demonstrate that appropriately designed random feature distributions can substantially improve approximation accuracy for several representative PDEs, including the Helmholtz equation, linear advection with relatively high wave speed, and the wave equation. These results suggest that kernel design plays a central role in the performance of random feature PDE solvers, although developing automatic kernel selection strategies remains an important direction for future research.

The main contributions of this paper are as follows:
\begin{itemize}
    \item[(I)] we employ an auto-differentiation based optimization framework using Adam and L-BFGS, making the proposed method readily applicable to a broad class of PDEs without manual Jacobian derivations;
    \item[(II)] we extend the theoretical analysis by establishing approximation error estimates in Sobolev $H^1$ spaces, providing stronger approximation guarantees for random feature models;
    \item[(III)] we introduce product random feature constructions that separately model spatial and temporal variables, enabling more flexible kernel design for time-dependent PDEs;
    \item[(IV)]  We present convergence studies together with comparisons against PINNs, self-adaptive PINNs, and Extreme Learning Machines on several representative benchmark problems, demonstrating the effectiveness of the proposed framework on several different kinds of PDEs.
\end{itemize}
The remainder of this paper is organized as follows. Section~\ref{sec:bg} reviews the necessary background on reproducing kernel Hilbert spaces, random feature methods, and physics-informed neural networks. Section~\ref{sec:main} presents the proposed physics-informed random feature framework together with the theoretical analysis. In Section~\ref{sec:numerics}, we present numerical experiments including numerical verification of approximation results and comparison on several benchmark PDEs.
Finally, we conclude in Section~\ref{sec:conclude} and discusses future research directions.
\subsection{Related Work}
The seminal paper~\cite{Raissi2019PINN} which introduced the basic concept of Physics-informed Neural Networks started a new era of scientific machine learning.  Since then, causality PINN~\cite{WANG2024116813}, self-adaptive PINN~\cite{MCCLENNY2023111722}, hard-constrained PINN~\cite{Hao2026}, time-marching PINN~\cite{lwight2021}, PINN with curriculumn training~\cite{3540261.3542294}, data-driven-viscosity PINN~\cite{COUTINHO2023112265}, domain and dimension co-decomposition PINN with MOE~\cite{shang2026dpinns, shang2026dimension}, have sprung up for improving the training difficulties when using PINN to solve various stiff PDEs and hyperbolic PDEs. A few more papers, such as~\cite{Shin2020Convergence, 10.1093/imanum/drab093,doumeche2025convergece,bonito2025convergence}, have been introduced to discuss the theoretical convergence and other properties of using PINN for solving PDEs.  The limited numbers are partly due to the immense difficulties in merging the PDE-theories into approximations provided by deep neural networks.

Kernel-based methods on the other hand can provide theoretical foundations, simple implementation, and competitive performance, which have been widely studied and applied in scientific computing, particularly for solving partial differential equations; see, for example, in~\cite{michael2017RKHS, CHEN2021kernelPDE, BATLLE2025113488,Jnathan2025PIKL}. Kernel method represents the target function in a RKHS, providing flexible, nonparametric approximation from scattered data. Besides the deterministic point of view, it is also equivalent to Gaussian process, which naturally enables uncertainty quantification~\cite{pfortner2022physics, Li2024Parameter, LEE2025GP}.

Existing kernel-based methods rely on inverting a large-scale kernel matrix, which is computationally expensive and limits the use of kernel methods for large-scale datasets. This motivates the development of scalable kernel surrogate methods, such as sparse Cholesky factorization~\cite{chen2025sparse} and random Fourier feature approximation~\cite{Rahimi2007RFM}. While the latter was originally proposed to approximate large-scale kernel machines, it can be viewed as a special case of randomized shallow neural networks where the input weights and biases are randomly sampled and fixed and only the weights in the readout layer are trainable. Randomized neural networks have also been used in scientific computing~\cite{DONG2021ELM,WANG2024Extreme, Liao2026PDE}, as they provide a simple implementation and an efficient training process, but their theoretical guarantees remain largely unexplored.
\subsection{Notation}
In this section, we denote the notation used throughout.  We let $\R$ and $\C$ be the set of real numbers and complex numbers. Denote $\irm=\sqrt{-1}$ the imaginary unit and $\overline{x}$ the complex conjugate of complex number $x$. Denote by $[m]={1,2,\dots,m}$. We denote vectors and matrices with boldface lowercase and uppercase letters, respectively. For two vectors $\xb,\yb\in\C^d$, the inner product is denoted by $\langle \xb,\yb\rangle = \sum_{i=1}^dx_i\overline{y}_i$. 
For a vector $\xb\in\R^d$, we denote by $\|\xb\|_p$ the $\ell^p$-norm of $\xb$. $\Ib_k$ is used to denote the identity matrix in $\R^{k\times k}$.  
For a domain $X\subset\R^d$, we denote $\|\cdot\|_{L^2(X)}$ the norm of $L^2(X)$ space (We omit the domain when there is no ambiguity).
Moreover, we denote $W^{s,p}(X)$ the Sobolev space consisting of $L^p(X)$ functions whose mixed partial derivatives up to order $s$ exist in the weak sense and are in $L^p(X)$. Its norm is denoted by $\|\cdot\|_{s,p}$. For the special case $p=2$, we use $H^s(X)$ to denote $W^{s,2}(X)$, equipped with norm $\|\cdot\|_{H^s}$ and inner product $\langle \cdot,\cdot\rangle_{H^s}$. We also denote $H^{s}_{\mix}(X)$ the Sobolev space of dominating mixed smoothness $s$ defined over domain $X$.
\section{Background}\label{sec:bg}
In this section, we review the theory of reproducing kernel Hilbert space (RKHS), kernel method, and random features method. We then introduce the formulation of physics-informed machine learning.
\subsection{Reproducing Kernel Hilbert Space and Kernel Method}\label{subsec:RKHS}
We first formalize a reproducing kernel Hilbert space in the following definition.
\begin{definition}
Let $\Hc$ be a Hilbert space of functions defined over set $X$.
We say that $\Hc$ is a {\bf reproducing kernel Hilbert space (RKHS)} if, for each $\xb\in X$, the point evaluation is a bounded linear functional, i.e. there exists some constant $M_\xb\geq0$ such that 
\begin{equation}
|f(\xb)| \leq M_\xb \|f\|_\Hc, \quad \mbox{for all } f\in\Hc.
\end{equation}
\end{definition}

In the context of PDEs, Sobolev spaces play a central role as the fundamental function spaces in which solutions and their weak derivatives can be rigorously defined and analyzed.
Moreover, Sobolev spaces are indeed RKHSs under some certain conditions. 
When $s>d/2$ ($d$ is the dimension of the domain), the Sobolev embedding theorem implies that point-evaluation functionals are continuous on $H^s(X)$. Consequently, $H^s(X)$ is a RKHS.

A RKHS induces a function of two inputs (kernel function) with the reproducing property. 
Following the Riesz representation theorem, there exists a function $K_\xb\in\Hc$ such that the following reproducing property holds
\begin{equation}
f(\xb) = \langle f,K_\xb\rangle_\Hc, \quad \mbox{for all } f\in\Hc.
\end{equation}
Since every point evaluation is a bounded linear functional, then we apply Riesz representation theorem one more time showing that there exists a function $K_{\xb'}\in\Hc$ such that
\begin{equation}
K_\xb(\xb') = \langle K_\xb, K_{\xb'}\rangle_\Hc. 
\end{equation}
Therefore, we obtain a {\bf kernel function} $K:X\times X\to\R$ such that $K(\xb,\xb') = \langle K_\xb,K_{\xb'}\rangle_\Hc$, which is called a reproducing kernel of $\Hc$.
For convenience, we define the canonical feature map $\phi:\xb\in X\mapsto K_\xb\in\Hc$ satisfying
\begin{equation}
K(\xb,\xb') = \langle K_\xb,K_{\xb'}\rangle_\Hc = \langle \phi(\xb), \phi(\xb')\rangle_\Hc.
\end{equation}
We define a positive (semi)definite kernel in the following definition. 
\begin{definition}
The kernel $K$ is a positive (semi)definite kernel, in the sense that, for any $m\geq1$ and any $\xb^{(1)},\dots,\xb^{(m)}\in X$, the $m\times m$ matrix with entries $K(\xb^{(i)}, \xb^{(j)}), i,j\in[m]$, is positive (semi)definite. 
\end{definition}
Above statement starts with a RKHS and then define its corresponding kernel function. 
On the other hand, starting with a positive semidefinite kernel $K$, there is a unique reproducing kernel Hilbert space with $K$ as its reproducing kernel. 
This results, known as the Moore-Aronszajn theorem, is stated in the following theorem.
\begin{theorem}[Moore-Aronszajn theorem \cite{Aronszajn1950RKHS}]
If $K: X\times X\to\R$ is a positive semidefinite kernel, then there exists a unique reproducing kernel Hilbert space $\Hc$ whose reproducing kernel coincides with $K$.
\end{theorem}

Kernel methods have been widely used in machine learning algorithms due to their ability to model nonlinear relationship through linear operations in a feature space \cite{scholkopf2001Learning, Hofmann2008Kernel}. Specifically, kernel methods first map samples to a high-dimensional, even infinite-dimensional, feature space (Hilbert space) through a feature map $\phi$, and then apply linear models in the high-dimensional feature space. While the computations are challenging in high-dimensional feature spaces, the corresponding kernel function $K$ indeed computes the inner product in the high-dimensional feature space using samples in the low-dimensional samples space, which significantly simplifies the computation. 

We briefly introduce the formulation of kernel method in regression problems. Suppose we observe $m$ sample pairs $(\xb_i,y_i)$, $i\in[m]$, where $\xb_i$ are the inputs in a compact domain $X\subset \R^d$ and $y_i\in\R$ are the responses (or labels). 
We assume that the responses are generated according to
\begin{equation}
y_i = f^*(\xb_i) + \epsilon_i, \quad \mbox{ for all } i\in[m],
\end{equation}
where $f^*$ is the unknown target function and $\epsilon_i$ are observational errors.
Suppose that the unknown target function $f^*$ lies in a RKHS $\Hc_K$ with positive definite kernel function $K:X\times X\to\R$.
We then define the kernel ridge estimator as
\begin{equation*}
\hat{f}_\lambda = \argmin_{f\in\Hc_K} \,\,\,\sum_{i=1}^m \left(f(\xb_i) - y_i\right)^2 + \lambda\|f\|^2_\Hc,
\end{equation*}
where $\lambda>0$ is the regularization parameter. As $\lambda\to0$, the kernel ridge estimator $\hat{f}_\alpha$ converges to the kernel ridgeless estimator, which is defined as
\begin{equation*}
\hat{f} = \argmin_{f\in\Hc_K} \,\,\, \|f\|_\Hc \quad \mbox{ subject to } f(\xb_i) = y_i, \quad \mbox{ for all } i\in[m].
\end{equation*}
We let $\Kb\in\R^{m\times m}$ be the kernel matrix defined element-wise as $\Kb_{i,j} = K(\xb_i,\xb_j)$ and $\yb$ be the vectors of responses $y_1,\dots,y_m$. The kernel representer theorem \cite{Wahba1990Spline, Scholkopf2001Representer} shows that the kernel ridge(less) estimator can be written as $\sum_{i=1}^m \hat{c}_iK(x,x_i)$, where the coefficient vector $\Hat{\cb}$ of coefficients $\hat{c}_1,\dots,\hat{c}_m$ has the following closed form: 
\begin{itemize}
\item Kernel ridge regression: $\Hat{\cb} = (\Kb+\lambda \Ib_m)^{-1}\yb$,
\item Kernel ridgeless regression: $\Hat{\cb} = \Kb^{-1}\yb$.
\end{itemize}
In practice, people usually use translation-invariant (shift-invariant) kernels, which is defined in the following definition.
\begin{definition}
A positive definite kernel $K:\R^d\times \R^d\to\R$ is {\bf translation-invariant} if there exists a function $\Tilde{K}:\R^d \to\R$ such that
\begin{equation*}
K(\xb,\xb') = \Tilde{K}(\xb-\xb') \quad \mbox{ for all } \xb,\xb'\in \R^d.
\end{equation*}
\end{definition}

Some commonly used translation-invariant kernels include 
\begin{itemize}
\item Gaussian Kernel: $K(\xb,\xb') = \exp\left(-\frac{\|\xb-\xb'\|_2^2}{2\sigma^2}\right)$, where $\sigma>0$ is the scale parameter.
\item Mat{\'e}rn Kernel: $K_\nu(\xb,\xb') = \Mc_{\nu,\sigma}(||\xb-\xb'\|_2)$, where the Mat{\' e}rn covariance function $\mathcal{M}_{\nu,\sigma}(x):\R\to\R$ with smoothness parameter $\nu>0$ and scale parameter $\sigma>0$ is defined as
\begin{equation*}
\mathcal{M}_{\nu,\sigma}(x) = \frac{2^{1-\nu}}{\Gamma(\nu)} \left( \frac{x}{\sigma} \right)^\nu B_\nu\left( \frac{x}{\sigma} \right)
\end{equation*}
and $B_\nu$ is a modified Bessel function of the second kind of order $\nu$.
\item Exponential Kernel\footnote{In some literature, this kernel is also called Laplace kernel. To avoid the confusion with the Laplace kernel defined via the $\ell^1$-norm, we refer to it as the exponential kernel in this paper.}: $K(\xb,\xb') = \exp\left(-\frac{\|\xb-\xb'\|_2}{\sigma}\right)$, where $\sigma>0$ is the scale parameter. It can be viewed as a special case of Mat{\' e}rn kernel with $\nu=1/2$.
\item Laplace kernel: $K(\xb,\xb') = \exp\left(-\frac{\|\xb-\xb'\|_1}{\sigma}\right)$, where $\sigma>0$ is the scale parameter.
\end{itemize}

The RKHS associated with the Mat{\' e}rn kernel $K_\nu(\xb,\xb')$ has the inner product
\begin{equation*}
\langle f,g\rangle_{\mathcal{M}_{\nu,\sigma}} = \int_{\mathbb{R}^d} \frac{\hat{f}(w)\overline{\Hat{g}(w)}}{(1+\|w\|^2)^{\nu+d/2}} dw
\end{equation*}
up to constants, which we recognize as the inner product of the classical Sobolev space $H^{\nu+d/2}(\mathbb{R}^d)$. 
For the general Mat{\' e}rn Kernel $\mathcal{M}_{\nu,\sigma}(x)$, we have some scaling constant and the corresponding RKHS is equivalent to the classical Sobolev space $H^{\nu+d/2}(\mathbb{R}^d)$.
Therefore, the smoothness parameter in the Mat{\'e}rn covariance function $\Mc_{\nu,\sigma}(x)$ determines the smoothness of the resulting functions. 
Following the fact that the Gaussian kernel is the limits of Mat{\'e}rn kernen as smoothness parameter $\nu\to\infty$ \cite{Porcu2024Matern}, we can show that the RKHS of Gaussian kernel is $H^{\infty}(\R^d)$.
For the Laplace kernel (with scale parameter $\sigma=1$), the corresponding RKHS is the Sobolev space of dominating mixed smoothness $H^{1}_{\mix}(\R^d)$. Unlike the Isotropic Sobolev space $H^{s}(\R^d)$, the mixed smoothness Sobolev space controls the each coordinate independently and the following embedding $H^{s}_{\mix}(\R^d)\subset H^s(\R^d)$ is continuous.
\subsection{Random Feature}
The major challenge of kernel method is that it scales poorly for large amount of sample pairs. 
It requires $\Oc(m^3)$ training time and $\Oc(m^2)$ storage space, which is computationally infeasible when sample size $m$ is large.
Therefore, random feature was originally proposed to approximate large-scale kernel machines \cite{Rahimi2007RFM}.
The random feature method provides a low-rank approximation of the kernel matrix, thereby reducing the computational cost and memory requirements associated with operating on the kernel matrix. 
The theoretical foundation of random (Fourier) feature builds on Bochner's Theorem \cite{bochner1955harmonic}, a classical result from harmonic analysis.
\begin{theorem}
A continuous translation-invariant kernel $K:X\times X\to\R$ satisfying $K(\xb,\xb') = \Tilde{K}(\xb-\xb')$ is positive definite if and only if $\Tilde{K}(\omegab)$ is the Fourier transform of a non-negative measure.
\end{theorem}
With Bochner's Theorem at hand, one can rewrite a translation-invariant kernel as an integral form 
\begin{equation}
K(\xb,\xb') = \int_{\R^d} \exp(\irm\langle \omegab,\xb-\xb'\rangle)d\rho(\omegab),
\end{equation}
where $\rho$ is a proper probability measure if the kernel is scaled properly. 
Approximating the above integral using Monte Carlo sampling yields the random Fourier feature approximation
\begin{equation*}
K(\xb,\xb') = \int_{\R^d} \exp(\irm\langle \omegab,\xb-\xb'\rangle)d\rho(\omegab) \approx \frac{1}{N}\sum_{j=1}^N \exp(\irm\langle \omegab_j,\xb-\xb'\rangle) = \langle \phi_N(\xb),\phi_N(\xb')\rangle,
\end{equation*}
where $\phi_N(\xb):X\to\R^N$ is a random finite-dimensional feature map defined as 
\begin{equation*}
\phi_N(\xb) = \frac{1}{\sqrt{N}}\Big[ \exp(\irm\langle\omegab_1,\xb\rangle), \,\,\dots\,\, , \exp(\irm\langle\omegab_N,\xb\rangle) \Big]^\top \in\C^N.
\end{equation*}
Here $\{\omegab_j\}_{j\in[N]}$ are independent samples from probability measure $\rho$. The feature map $\phi_N$ depends on random samples $\{\omegab_j\}_{j\in[N]}$, while we omit the dependency in order to simplify the notation.
To obtain a real-valued random feature approximation for kernel $K$, one can introduce a random phase $b\sim \rm{Uniform}[0,2\pi]$, yielding the standard random cosine features representation in \cite{Rahimi2007RFM}. Specifically, one can define the feature map as
\begin{equation*}
\phi_N(\xb) = \sqrt{\frac{2}{N}}\Big[ \cos(\langle\omegab_1,\xb\rangle+b_1), \,\,\dots\,\, , \cos(\langle\omegab_N,\xb\rangle+b_N) \Big]^\top \in\R^N, 
\end{equation*}
where $\{\omegab_j\}_{j\in[N]}$ are independent samples from probability measure $\rho$ and $\{b_j\}_{j\in[N]}$ are independent samples from Uniform distribution on $[0,2\pi]$.
The random feature map induces a translation-invariant kernel
\begin{equation*}
K_N(\xb,\xb') = \langle \phi_N(\xb),\phi_N(\xb')\rangle,
\end{equation*}
which converges to $K(\xb,\xb')$ as $N\to\infty$.
For later use, and with a slight abuse of notation, we introduce $\varphi(\xb;\omegab)$ denote a generic random feature associated with $\omegab$. Depending on the context, $\varphi(\xb;\omegab)$ may refer either to a complex-valued feature $\exp(\irm\langle\omegab,\xb\rangle)$ or a real-valued feature $\cos(\langle\omegab,\xb\rangle+b)$ \footnote{Precisely, we consider the augmented vectors $\widetilde{\omegab} = [\omegab,b]$ and $\widetilde{\xb}=[\xb,1]$. The generic real-valued feature is indeed $\varphi(\widetilde{\xb};\widetilde{\omegab})$. To simplify the notation, we omit the tildes when no confusion can arise.}.

Bochner's theorem establishes a one-to-one correspondence between continuous translation-invariant positive definite kernels and their spectral densities. Here, we summarize the commonly used kernels and their corresponding spectral densities.
\begin{itemize}
\setlength{\leftskip}{-0.5cm}
\item {\bf Gaussian kernel} with scale parameter $\sigma$. The corresponding spectral density is Gaussian with mean 0 and variance $1/\sigma^4$.
\item {\bf Mat{\'e}rn kernel} with smoothness parameter $\nu$ and scale parameter $\sigma$. The corresponding probability distribution is general t-distribution with degree of freedom $df=2\nu$, location parameter $\mu=0$, and scale parameter $1/(2\pi\sigma)$.
\item {\bf Laplace kernel} with scale parameter $\sigma$ corresponds to tensor-product Cauchy distribution with scale parameter $1/\sigma$, whose probability density function with scale parameter $\gamma$ is
\begin{equation*}
f(\xb;\gamma) = \left(\frac{2}{\pi}\right)^d\prod_{i=1}^d \frac{\gamma}{\gamma^2+x_j^2}.   
\end{equation*}
\end{itemize}

A modern interpretation of random feature method is viewing it as a randomized neural network. 
We recall the generic real-valued random cosine feature $\varphi(\xb,\omegab) = \cos(\langle\omegab,\xb\rangle+b)$ where $\omegab$ is sampled from the spectral distribution $\rho$ associated with the kernel and $b\sim \rm{Uniform}[0,2\pi]$.
We are interested in approximating mixture of the form $f(\xb) = \int_{\R^d}\alpha(\omegab)\varphi(\xb,\omegab)d\rho(\omegab)$ by a finite sum of the form
\begin{equation}
\label{RF:finite_sum}
\Hat{f}(\xb) = \sum_{j=1}^N \Hat{c}_j \varphi(\xb,\omegab_j) = \sum_{j=1}^N \Hat{c}_j \cos(\langle\omegab_j,\xb\rangle+b_j).    
\end{equation}
Here the parameters $\{\omegab_j\}_{j\in[N]}$ are randomly sampled and kept fixed. 
The resulting finite sum can be viewed as a single hidden layer neural network with cosine activation function. 
Unlike the vanilla neural networks, the hidden-layer weights $\omegab_j$ and bias terms $b_j$ are sampled and then fixed, whereas only the output-layer coefficients $\Hat{c}_j$ are learned.  
This randomized construction offers several advantages. First, it significantly reduces the number of trainable parameters since only the output-layer coefficients can be trained.
Second, the training process involves solving a convex optimization problem, often reduced to linear regression, thereby avoiding the difficulties in solving non-convex optimization problems. 
Despite this simplification, randomized neural networks retain strong approximation capabilities and have been successfully applied to a variety of scientific problems, including operator learning \cite{Liao2025Cauchy, FABIANI2025RandONets} and dynamical system \cite{Liao2026LR,Liao2026Koopman}. 
In later sections, we investigate their approximation properties. To this end, we first introduce the function space in which the approximation analysis will be carried out.
Let's fix a probability measure $\rho$ (associated with a kernel $K$), we define the following function space:
\begin{equation}
\label{Function_space}
\Fc(\rho) := \left\{f(\xb)=\int_{\R^d}\alpha(\omegab)\varphi(\xb,\omegab)d\rho(\omegab): \|f\|^2_{\Fc(\rho)} = \|\alpha\|^2_{L^2(\rho)} <\infty \right\}.
\end{equation}
It is well-known that the completion of $\Fc(\rho)$ is a RKHS and coincides with the RKHS associated with kernel $K$, see \cite{Rahimi2008RFM} for a complete proof. 
\subsection{Physics-Informed Neural Networks (PINNs)}
In this section, we introduce Physics-Informed Neural Networks (PINNs), which are a data-driven scientific machine learning technique for solving Partial Differential Equations (PDEs). 
The main idea of PINN is approximating PDE solutions by training a neural network to minimize a loss function including terms reflecting initial and boundary conditions and PDE residual at a set of collocation points.
Consider the following general PDE problem with both initial condition (IC) and boundary condition (BC):
\begin{subequations}
\begin{align}
& \Pc[u(\xb,t)] = 0, \quad (\xb,t)\in D \times [0,T] \label{PDE:govern} \\
& u(\xb,0) = u_0(\xb), \quad \xb\in D \label{PDE:init}\\
& \Bc[u(\xb,t)] = 0, \quad (\xb,t) \in \partial D \times[0,T] \label{PDE:bd},
\end{align}
\end{subequations}
where $u(\xb,t)$ is unknown, $D$ is the problem domain, $\xb$ is a spatial vector variable, $t$ is time, $\Pc[\cdot]$ is an interior differential operator, and $\Bc[\cdot]$ is a boundary differential operator. Equations \eqref{PDE:init} and\eqref{PDE:bd} provide initial and boundary conditions, respectively.
A PINN approximates $u(\xb,t)$ by the output of a neural network $u_\theta^\sharp(\xb,t)$ with trainable parameters $\theta$ (weights and bias).
The loss function used to train neural network parameters incorporates the physics embedded in the PDE, initial, and boundary conditions \footnotetext{In Eq~\eqref{PINN_loss}, notation $\Pc[u_\theta^\sharp(\xb_i,t_i)]$ means that we apply the differential operator $\Pc$ to function $u_\theta^\sharp$ first and we then evaluate the function values at points $(\xb_i,t_i)$. The same convention is applied to $\Bc[u_\theta^\sharp(\xb_i,t_i)]$.}:
\begin{equation}
\label{PINN_loss}
\Lc(\theta) =  \underbrace{\frac{\lambda_1}{M_r}\sum_{i=1}^{M_r} \Pc[u_\theta^\sharp(\xb_i,t_i)]^2}_{\Lc_{PDE}:\,\,\mbox{PDE redisual}} + \underbrace{\frac{\lambda_2}{M_0} \sum_{i=1}^{M_0} \left(u_\theta^\sharp(\xb,0) - u_0(\xb) \right)^2}_{\Lc_0:\,\,\,\mbox{IC residual}} + \underbrace{\frac{\lambda_3}{M_b}\sum_{i=1}^{M_b} \Bc[u_\theta^\sharp(\xb_i,t_i)] ^2 }_{\Lc_b:\,\,\, \mbox{BC residual}},
\end{equation}
where $\{(\xb_i,t_i)\}_{i=1}^{M_r}$ is a set of collocation points where the governing PDE \eqref{PDE:govern} is enforced, $\{\xb_i\}_{i=1}^{M_0}$ is a set of points in the domain to enforce the initial condition \eqref{PDE:init}, $\{\xb_i,t_i\}_{i=1}^{M_b}$ is a set of points to enforce the boundary condition \eqref{PDE:bd}, and $\lambda_1,\lambda_2,\lambda_3>0$ are pre-selected regularization parameters.
These collocation points are randomly sampled in the PDE, initial, and boundary conditions domains. 
The regularization parameters are chosen to balance PDE, IC, and BC residuals.
\section{Main Results}\label{sec:main}
In this section, we present our proposed physics-informed random feature networks and the associated computational framework, followed by theoretical results on their approximation properties.

\subsection{Physics-informed Random Feature Network}
In this section, we present our proposed physics-informed random feature network for solving PDEs. We first consider the following general PDE problem:
\begin{equation}
\label{PDE}
\begin{aligned}
\Pc[u](\xb) &= 0, && \xb\in D \\
\Bc[u](\xb) &= 0, && \xb \in \partial D, 
\end{aligned}
\end{equation}
where $D\subset\R^d$ is the domain with the boundary $\partial D$, $\Pc$ is the interior differential operator, and $\Bc$ is the boundary differential operator.
For the sake of brevity, we assume that the PDE is well-defined pointwise and has a unique strong solution throughout this paper.
We propose to solve the PDE problem \eqref{PDE} by using a random feature network.
Specifically, we model the solution to the PDE problem as a finite sum of the form \eqref{RF:finite_sum}, i.e.
\begin{equation}
\label{RF_form}
u_\cb^\sharp(\xb) = \sum_{k=1}^N c_k^\sharp \phi( \xb, \omegab_k),
\end{equation}
and then impose the physics constraints using collocation points.  
Let $\{\xb_j\}_{j\in[M]}$ be a set of collocation points such that $\{\xb_j\}_{j\in[M_D]}$ is a set of interior points and $\{\xb_j\}_{j=M_D+1}^M$ is a set of boundary points. 
We randomly draw features $\{\omegab_k\}_{k\in[N]}$ from a known distribution $\rho(\omegab)$. 
We learn the output-layer coefficients by solving the following optimization problem: 
\begin{equation}
\label{MinnormRF_PDE}
\begin{aligned}
& \minimize_{\cb\in\R^N} && \|\cb\|_2^2 \\
& \mbox{s.t.} && \Pc[u_\cb^\sharp](\xb_j) = 0, \quad \mbox{ for } j=1,\dots,M_\Omega \\
& && \Bc[u_\cb^\sharp](\xb_j) = 0, \quad \mbox{ for } j=M_\Omega+1,\dots,M
\end{aligned}
\end{equation}
For the linear PDEs, we can write the constraints in \eqref{MinnormRF_PDE} as a linear system $\Ab\cb = \yb$. Then the optimization problem \eqref{MinnormRF_PDE} reduces to a constrained min-norm optimization problem, which has explicit solution $\cb = \Ab^\dagger\yb$, where $\Ab^\dagger$ is the pseudo-inverse of matrix $\Ab$. 
For general nonlinear PDE problems, the governing PDE and boundary conditions are incorporated into the loss function as soft constraints. Specifically, we solve the following optimization problem
\begin{equation}
\label{reg_RF_PDE}
\minimize_{\cb\in\R^N} \|\cb\|_2^2 + \lambda_1 \sum_{j=1}^{M_\Omega}\left(\Pc[u_\cb^\sharp](\xb_j)\right)^2 + \lambda_2 \sum_{j=M_\Omega+1}^M \left(\Bc[u_\cb^\sharp](\xb_j)\right)^2,
\end{equation}
where $\lambda_1,\lambda_2>0$ are regularization parameters. When $\lambda_1,\lambda_2\to0$, the solution of \eqref{reg_RF_PDE} converges to the solution of \eqref{MinnormRF_PDE}. 
\begin{remark}
The loss function in \eqref{reg_RF_PDE} is slightly different from the PINN loss \eqref{PINN_loss} due to the term $\|\cb\|_2^2$. It can be viewed as a regularization term to control the model complexity. One can also remove this term in practice.
\end{remark}

We now consider the a PDE problem with both initial and boundary conditions Eq~\eqref{PDE:govern}-\eqref{PDE:bd}, and discuss how we design a random feature network to approximate its solution.
When dealing with PDEs involving both spatial and temporal variables, a standard approach is to concatenate them into a single input vector, i.e. ($\xb,t$) \cite{Raissi2019PINN, Liao2026PDE}.  
When we use the random feature network \eqref{RF_form} proposed above, it implicitly assumes that spatial and temporal variables are treated identically and the random features are sampled from the same distribution across spatial and temporal space.
However, it is generally inappropriate because space and time may have fundamentally different physical nature, interpretation, and characteristic scales. 

To alleviate this issue, we propose the following random feature network:
\begin{equation}
\label{RF_IC_BC}
u_\cb^\sharp(\xb,t) = \sum_{i=1}^{N_t}\sum_{k=1}^{N_x} c_{i,k} \varphi(\xb, \omegab_k)\psi(t,\xi_i),
\end{equation}
where $\{\omegab_k\}_{k=1}^{N_x}$ and $\{\xi_i\}_{i=1}^{N_t}$ are randomly sampled from distributions $\rho_1$ and $\rho_2$ and kept fixed. 
This idea is motivated by approximating a product kernel, which allows separated treatment of spatial and temporal components while still modeling their joint interactions.
Specifically, we define a product kernel via two translation-invariant positive definite kernels $K_1$ and $K_2$
\begin{equation*}
K((\xb,t),(\xb',t')) = K_1(\xb,\xb') K_2(t,t').
\end{equation*}
Notice that the product kernel is also translation-invariant and positive definite.
We apply random feature approximation for $K_1$ and $K_2$ separately and obtain 
\begin{equation*}
\begin{aligned}
K_1(\xb,\xb') K_2(t,t') \approx& \langle \phi_{N_x}(\xb),\phi_{N_x}(\xb')\rangle \langle\psi_{N_t}(t),\psi_{N_t}(t')\rangle \\
=& \langle \phi_{N_x}(\xb)\otimes\psi_{N_x}(t), \phi_{N_t}(\xb')\otimes\psi_{N_t}(t')\rangle,
\end{aligned}
\end{equation*}
where $\otimes$ denote the Kronecker product. 
Based on this separable random feature construction, we directly parametrize $u^\sharp(\xb,t)$ in Eq~\eqref{RF_IC_BC} as a finite sum over spatial and temporal random features. 
In practice, the spatial variable $\xb$ and the temporal variable $t$ may use different numbers of random features ($N_x$ and $N_t$) and different sampling distributions ($\rho_1$ and $\rho_2$), which makes the random feature networks more expressive and better capable capturing spatial-temporal behaviors. A numerical comparison for the product and uniform random feature models using the linear advection-diffusion equation is provided in Section~\ref{sec:LADE}.

In addition, the product formulation can also be applied to the case where the PDE variables have different scales along different spatial directions.
We may use different random features for different spatial variables to capture its behavior, i.e
\begin{equation}
u_\cb^\sharp(\xb_1,\xb_2) = \sum_{i=1}^{N_t}\sum_{k=1}^{N_x} c_{i,k} \varphi(\xb_1, \omegab_k)\psi(\xb_2,\xi_i).
\end{equation}
We refer readers to Section~\ref{sec:Helm} for a numerical example of different features for different spatial directions.

\subsection{Universal Approximation}

In this section, we show that randomized neural networks are universal approximators of continuous functions. 
Universal approximation results for this class of networks were established in \cite{FABIANI2025RandONets}. We summarize the main results below.
\begin{theorem}[Proposition 1 in \cite{FABIANI2025RandONets}]
Let $K\subset \R^d$ compact and $U \subset C(K)$ compact and consider a parametric family of random activation function $\{\varphi(\xb,\omegab):\xb\in\R^d, \omegab\in \Omega \}$, where $\omegab\in\Omega$ is a vector of randomly chosen (hyper)parameters, and assume that $\phi$ are uniformly bounded in $\R^d\times \Omega$. Let $\rho$ be a probability distribution on $\Omega$. Given any $\epsilon$, there exists a $N\in\N$ and i.i.d samples $\{\omegab_k\}_{k\in[N]}$ from $\rho$, chosen independently of $f$, such that for every $f\in U$ the random approximation
\begin{equation*}
    f_\epsilon(\xb) = \sum_{k=1}^N c_k[f]\varphi(\xb,\omegab_k),
\end{equation*}
approximates $f$ in the sense that with high probability 
\begin{equation*}
\|f - f_\epsilon\|_{L^2(\mu)} < \epsilon,
\end{equation*}
for a suitable probability measure $\mu$ over $K$. Moreover, if $\varphi(\xb,\omegab) = \psi(\xb\cdot\omegab)$, for a L-Lipschitz function $\psi$, the above approximation is uniform (i.e. in the supremum norm).
\end{theorem}
\begin{proof}[Sketch of Proof]
The proof in \cite{FABIANI2025RandONets} follows three steps: first, it applies universal approximation results of vanilla shallow neural networks with a continuous sigmoidal (non-polynomial) activation function; second, it uses radial basis function (RBF) networks to approximate the continuous sigmoidal activation function, and hence shows that RBF networks are universal approximators; third, it shows that randomized neural networks can approximate RBF networks, and then randomized neural networks are universal approximators of continuous functions.

Alternatively, we can apply universal kernels \cite{Micchelli2006Universal} to show that their corresponding RKHS is dense in $C(K)$, and then show that randomized neural networks (random features models) can approximate any functions in the RKHS. Therefore, randomized neural networks are universal approximators of continuous functions.  
\end{proof}
We mainly focus on Gaussian and Mat{\'e}rn kernels in the numerical experiments, both of which are universal kernels. The corresponding approximation results are special cases of a more general theory for translation-invariant kernels. We refer the reader to Section~4 in \cite{Micchelli2006Universal} for more details.

\subsection{Approximation Error}

In this section, we study the approximation of RKHS functions using random features and establish bounds on the corresponding approximation error. Specifically, we first state the results of approximation error in $L^\infty$ and $L^2$ norms developed in \cite{Liao2026PDE}. We then derive the bounds on the approximation error in the Sobolev $H^1$ norm. 

\begin{theorem}[Theorem 3 in \cite{Liao2026PDE}]
\label{Thm:L2}
Let $f$ be a function from $\Fc(\rho)$. Suppose that the random feature map $\phi$ satisfies $|\varphi(\xb,\omegab)|\leq1$ for all $\xb\in X$ and $\omegab\in\R^d$. Then for any $\delta\in(0,1)$, there exists $c^\sharp_1,\dots,c^\sharp_N$ so that the function
\begin{equation}
\label{RF_approximator}
f^\sharp(\xb) = \sum_{k=1}^N c^\sharp_k\varphi(\xb,\omegab_k) 
\end{equation}
satisfies
\begin{equation*}
\left| f(\xb) - f^\sharp(\xb)\right| \leq \frac{12\|f\|_{\Fc(\rho)}\log(2/\delta)}{\sqrt{N}}     
\end{equation*}
with probability at least $1-\delta$ over $\omegab_1,\dots,\omegab_N$ drawn i.i.d from $\rho(\omegab)$.
\end{theorem}
\begin{proof}[Sketch of Proof]
We outline the main idea of the proof. The proof is constructive. We first construct a random feature approximation $f^\sharp$ associated with the target function $f$. Since the true function is generally unknown to us, this construction is used for analysis and is not intended to be implemented directly. We then apply concentration inequalities to obtain the desired results.  
\end{proof}

Consequently, on any domain $X$ of finite Lebesgue measure $|X|$, the above uniform error bound immediately implies an $L^2$ error bound of the same order, i.e. 
\begin{equation*}
\|f-f^\sharp\|_{L^2(X)} \leq \frac{12\|f\|_{\Fc(\rho)}\log(2/\delta)\sqrt{|X|}}{\sqrt{N}}.    
\end{equation*}

With the $L^2$ error bound established, it remains to bound $\|\nabla(f-f^\sharp)\|_{L^2}$. Combining the resulting gradient estimate with the $L^2$ error bound yields an approximation error bound in $H^1$ norm. The main result will be summarized in Theorem~\ref{Thm:H1}.
Before we state the theorem, we introduce another function space consisting of functions whose regularity is controlled by the $L^2$ norm of the feature gradients. 
Specifically, let
\begin{equation*}
M^2:=\int_{\R^d} |\alpha(\omegab)|^2 \|\nabla\varphi(\xb,\omegab)\|_{L^2}^2d\rho(\omegab),
\end{equation*}
we then define the function space as
\begin{equation*}
\widetilde{\Fc}(\rho) = \left\{ f(\xb)=\int_{\R^d}\alpha(\omegab)\varphi(\xb,\omegab)d\rho(\omegab) \,\Big| \, M^2<\infty  \right\}.
\end{equation*}

\begin{theorem}[Approximation Error in $H^1$ norm]
\label{Thm:H1}
Let $f$ be a function from $\Fc(\rho)\cap\widetilde{\Fc}(\rho)$. Suppose that the random feature map $\phi$ satisfies $|\varphi(\xb,\omegab)|\leq1$ for all $\xb\in X$ and $\omegab\in\R^d$. 
We define a set 
\begin{equation*}
\epsilon_N(\delta) = \Big\{ \max_{k\in [N]} \max_{i\in[d]} \sup_{\xb\in X} \left| \frac{\partial \varphi(\xb, \omegab_k)}{\partial x_i} \right| \leq B_N(\delta) \Big\},
\end{equation*}
where $B_N(\delta)>0$ depends on $N$ and $\delta$. We assume that $\Pbb(\epsilon_N(\delta)) \geq 1-\delta/2$.
Then for any $\delta\in(0,1)$, there exists $c^\sharp_1,\dots,c^\sharp_N$ so that the function
\begin{equation}
f^\sharp(\xb) = \sum_{k=1}^N c^\sharp_k\varphi(\xb,\omegab_k) 
\end{equation}
satisfies
\begin{equation*}
\| f(\xb) - f^\sharp(\xb)\|_{H^1(X)} \leq \frac{12\|f\|_{\Fc(\rho)}\log(2/\delta)\sqrt{|X|}}{\sqrt{N}} + \frac{12MB_N(\delta)\|f\|_\Fc(\rho)\log(2/\delta)\sqrt{|X|}}{\sqrt{N}}
\end{equation*}
with probability at least $1-\delta$ over $\omegab_1,\dots,\omegab_N$ drawn i.i.d from $\rho(\omegab)$ and over random event $\epsilon_N(\delta)$.
\end{theorem}
\begin{proof}
We first introduce notation $\alpha_{\leq T}(\omegab) = \alpha(\omegab)\indicator_{|\alpha(\omegab)|\leq T}$ and $\alpha_{>T} = \alpha(\omegab) - \alpha_{\leq T}(\omegab)$ for any $T>0$.
We draw random features $\{\omegab\}_{k=1}^N$ from $\rho(\omegab)$. We then define $c_k^\sharp = \alpha_{\leq T}(\omegab_k) / N$ for all $k\in[N]$ and a random feature approximation as
\begin{equation*}
f^\sharp(\xb) = \frac{1}{N}\sum_{k=1}^N  \alpha_{\leq T}(\omegab_k) \varphi(\xb,\omegab).
\end{equation*}
For each component $x_i$ of $\xb\in\R^d$, we decompose the error into two terms, i.e.
\begin{equation*}
\left| \frac{\partial(f-f^\sharp)}{\partial x_i}\right| \leq \underbrace{\left| \frac{\partial(f-g)}{\partial x_i}\right|}_{I_1} + \underbrace{\left| \frac{\partial(g-f^\sharp)}{\partial x_i}\right|}_{I_2},
\end{equation*}
where $g=\Eb_{\omegab} f^\sharp$. Moreover, we notice that $f-g = \Eb_{\omegab} \alpha_{>T}(\omegab)\varphi(\xb,\omegab)$.

We first bound term $I_1$, which can be written as
\begin{equation*}
I_1 = \left| \Eb_{\omegab} \alpha_{>T}(\omegab) \frac{\partial \varphi(\xb,\omegab)}{\partial x_i}\right| = \left|\int_{\R^d} \alpha_{>T}(\omegab) \frac{\partial \varphi(\xb,\omegab)}{\partial x_i} d\rho(\omegab) \right|.
\end{equation*}
Then, we apply Cauchy-Schwarz inequality to obtain 
\begin{equation*}
I_1 \leq \underbrace{\left[ \int_{\R^d} \left|\alpha(\omegab) \frac{\partial \varphi(\xb,\omegab)}{\partial x_i} \right|^2 d\rho(\omegab) \right]^{1/2}}_{M_i} \left[ \int_{\R^d} \indicator_{|\alpha(\omegab)|\geq T}^2d\rho(\omegab) \right]^{1/2}.
\end{equation*}
Finally, applying Markov's inequality yields
\begin{equation*}
I_1 \leq M_i \frac{\Eb_{\omegab}|\alpha(\omegab)|^2}{T} = \frac{M_i \|f\|^2_{\Fc(\rho)}}{T}.
\end{equation*}

Next, we bound term $I_2$. By assumption, we have
\begin{equation*}
\Pbb(\epsilon_N(\delta)) \geq 1-\delta/2.
\end{equation*}
For any $\xb\in X$, we define random variable $Z(\omegab) = \alpha_{\geq T}(\omegab)\frac{\partial \varphi(\xb,\omegab)}{\partial x_i}$ and let $Z_1,\dots, Z_N$ be N i.i.d copies of $Z(\omegab)$ defined by $Z_k = Z(\omegab_k)$ for $k\in[N]$.
On the event $\epsilon_N(\delta)$, we have, with probability at least $1-\delta/2$,
\begin{equation*}
|Z_k| \leq T B_N(\delta).
\end{equation*}
The variance of $Z$ is bounded above as
\begin{equation*}
\Eb_{\omegab} |Z - \Eb_{\omegab} Z|^2 \leq \Eb_{\omegab} |Z|^2 \leq M_i^2.
\end{equation*}
By Lemma~\ref{lemma:bound} and Theorem~\ref{Thm:Bern}, it holds that, 
\begin{equation*}
I_2 = \left|\frac{1}{N}\sum_{k=1}^N Z_k - \Eb_{\omegab}Z\right|\leq \frac{4TB_N(\delta)\log(2/\delta)}{N} + \sqrt{\frac{2M_i^2\log(2/\delta)}{N}}.
\end{equation*}
Choosing $T=\sqrt{N}\|f\|_{\Fc(\rho)}$ and adding the bounds of $I_1$ and $I_2$ give
\begin{equation*}
\begin{aligned}
\left| \frac{\partial(f-f^\sharp)}{\partial x_i}\right| \leq &\frac{M_i\|f\|_{\Fc(\rho)}}{\sqrt{N}} + \frac{4M_i\|f\|_{\Fc(\rho)}B_N(\delta)\log(2/\delta)}{\sqrt{N}} + \frac{\sqrt{2\log(2/\delta)}M_i}{\sqrt{N}} \\
\leq& \frac{12M_iB_N(\delta)\|f\|_{\Fc(\rho)}\log(2/\delta)}{\sqrt{N}}.
\end{aligned}
\end{equation*}
Then we immediately obtain the following bound since $X$ has finite Lebesgue measure and we have the relation $M^2 = \sum_{i=1}^dM^2_i$ by the definitions of $M$ and $M_i$, i.e.
\begin{equation}
\label{grad_L2}
\|\nabla(f-f^\sharp)\|_{L^2(X)} = \sqrt{ \sum_{i=1}^d \left\| \frac{\partial (f-f^\sharp)}{\partial x_i} \right\|^2_{L^2(X)} } \leq \frac{12MB_N(\delta)\|f\|_\Fc(\rho)\log(2/\delta)\sqrt{|X|}}{\sqrt{N}}.
\end{equation}
Combining with the $L^2$ error bound in Theorem \ref{Thm:L2} leads to the desired bound in $H^1$ norm.
\end{proof}

We end this section with a few remarks. 
\begin{remark}
Consider the following 1D example where $f(x)$ is represented by its Fourier transform. Then, we have 
\begin{equation*}
\alpha(\omega) = \frac{\Hat{f}(\omega)}{\rho(\omega)},  \,\,\,\, \varphi(x,\omega)=\exp(ix\omega), \,\,\, \mbox{ and } M^2 = \int_{\R} \frac{|\hat{f}(\omega)|^2}{\rho(\omega)}d\omega.
\end{equation*}
Therefore, we show that the Fourier transform is controlled by density function of random feature. If we further take Cauchy distribution with density function $\rho(\omega) = \frac{1}{1+\omega^2}$, then we observe that $M$ is equivalent to $H^1$ norm.
The norm $M$ appearing in the definition of space $\widetilde{\Fc}(\rho)$ quantifies the regularity of functions and controls the approximation error. 
This approximation result has the same flavor as classical error estimates for finite element method (FEM): controlling the approximation error in $H^1$ norm requires a higher-order regularity assumption on the target function, typically the $H^2$ norm is finite.
\end{remark}

\begin{remark}
In the proof of Theorem \ref{Thm:H1}, we construct a random feature approximation $f^\sharp$ and select a positive constant $T=\sqrt{N}\|f\|_{\Fc(\rho)}$, which are the same those used in the proof of Theorem \ref{Thm:L2}. Therefore, we can use the results of Theorem \ref{Thm:L2} directly.
\end{remark}

\begin{remark}
In the theorem statement, we define a good event $\epsilon_N(\delta)$ such that the partial derivative of random feature is bounded with probability at least $1-\delta/2$.  
Good event $\epsilon_N(\delta)$ and quantity $B_N(\delta)$ depend on the choice of random features. 
Here, we derive an explicit $B_N(\delta)$ for Gaussian distribution $\Nc(0,\sigma^2)$ because we mainly use Gaussian distribution in our numerical experiments. 
Suppose that $\omega_i$ follows Gaussian distribution $\Nc(0,\sigma^2)$, then we have the tail bound 
\begin{equation*}
\Pbb(|\omega_i|\geq B) \leq 2\exp\left(-\frac{B^2}{2\sigma^2}\right).
\end{equation*}
Applying the union bound gives
\begin{equation*}
\Pbb(\max_{k\in[N]} \max_{i\in[d]} |\omega_{k,i}|\leq B) \geq 1-2Nd\exp\left(-\frac{B^2}{2\sigma^2}\right).
\end{equation*}
Since we assume that the failure probability is at most $\delta/2$, then we have
\begin{equation*}
B_N(\delta) = \sigma\sqrt{2\log\left(\frac{4Nd}{\delta}\right)}.
\end{equation*}
\end{remark}
\section{Numerical Experiments}\label{sec:numerics}
In this section, we evaluate the performance of the proposed physics-informed random feature network (PIRFN) model on several PDE problems. 
First, we numerically verify the approximation error decay rates derived in Theorems \ref{Thm:L2} and \ref{Thm:H1}. 
We then compare our proposed PIRFN model defined in \eqref{RF_IC_BC} with the uniform random feature model proposed in \cite{Liao2026PDE} to demonstrate its superiority in approximating PDE solutions, especially when spatio-temporal variables have different scales.
Finally, we compare with some state-of-the-art methods, including Physics-informed Neural Network (PINNs), self-adaptive Physics-informed Neural Network (SA-PINN), and extreme learning machine (ELM), across several benchmark PDE problems.
For all time-dependent numerical examples, we consider the random Fourier features of the form \eqref{RF_IC_BC} with cosine activation function for both spatial and temporal variables. The weights for both of them are independently sampled form Gaussian distributions and the choice of variances will be specified for each PDE problem. 
All source code will be uploaded to a Github repository and available upon completion of the revision.

\subsection{Measure of Accuracy}
We take several metrics to measure the model accuracy.
For the numerical verification of decay rate in Section \ref{sec:decay}, we consider $L^2$ error $\|u-\hat{u}\|_{L^2}$ and $H^1$ error $\|u-\hat{u}\|_{H^1}$ between true solution $u$ and the random feature approximation $\hat{u}$. 
Since we cannot evaluate the norms exactly, we use the Monte Carlo method to estimate them on a withheld set of discrete points $\{x_i\}_{i=1}^M$, see formulas below
\begin{equation*}
\begin{aligned}
\| u -\hat{u}\|_{L^2} &\approx \sqrt{\frac{1}{M}\sum_{i=1}^{M}\left|u(x_i)-\hat{u}(x_i)\right|^2}, \quad \text{and} \\
\| u -\hat{u}\|_{H^1} &\approx \sqrt{
\frac{1}{M}
\sum_{i=1}^{M}
\left[
\left|u(x_i) - \hat{u}(x_i)\right|^2
+
\left\|\nabla u(x_i) - \nabla\hat{u}(x_i)\right\|^2
\right]}.
\end{aligned}
\end{equation*}
For the comparison with other methods in Sections~\ref{sec:LADE}--\ref{sec:wave}, we assess the accuracy of each method using relative $L^2$ error and $L^\infty$ error. Similarly, we evaluated the errors using Monte Carlo method on a withheld set of test points $\{x_i\}_{i=1}^M$. The discrete formulas are given below
\begin{equation*}
\frac{\|u-\hat{u}\|_{L^2}}{\|u\|_{L^2}} \approx \frac{
\sqrt{\sum_{i=1}^{N}\left(\hat{u}(x_i)-u(x_i)\right)^2}}{\sqrt{\sum_{i=1}^{N}u(x_i)^2}}, \quad \|u-\hat{u}\|_{L^\infty} \approx \max_{i\in[1:M]}|u(x_i)-\hat{u}(x_i)|.
\end{equation*}
\subsection{Training Strategy}
For our proposed random feature-based method, the trainable coefficients are optimized using the Adam and L-BFGS optimizers implemented in PyTorch. For the Adam optimizer, we include a regularization term involving the norm of the coefficients in the loss function, as defined in \eqref{reg_RF_PDE}. In contrast, when using the L-BFGS optimizer, we minimize only the PDE residual loss without the additional regularization term. Based on our numerical experiments, these choices lead to better performance for the corresponding optimization methods. For simplicity, we set the regularization parameters $\lambda_1$ and $\lambda_2$ to one.

For the comparison methods, including PINN, SA-PINN, and ELM, we use the same PDE residual loss and implement their respective training procedures in PyTorch. For a fair comparison, the neural network-based methods are trained using the same two-stage optimization procedure, consisting of Adam followed by L-BFGS, with a comparable number of features to that used in our PRF method. For ELM, the hidden-layer weights $\omegab$ are randomly sampled from Gaussian distribution with covariance matrix $\frac{1}{2}\Ib$ and bias term $b$ are randomly sampled form Uniform distribution $Unif[-1,1]$.
Then they are fixed throughout the training process, while only the output-layer coefficients are optimized. Unless otherwise specified, the same stopping criteria are used across all methods. We note that, for the numerical examples below, the smaller number of L-BFGS iterations used by some comparison methods is primarily due to their errors reaching a plateau at an early stage of the optimization. We report the average results over ten independent trials, with different sets of training points used in each trial. For each trial, the training points are randomly sampled, while the test points are fixed at the grid points for each problem.

\subsection{Error Decay Rate}
\label{sec:decay}
We first numerically investigate the approximation errors established in Theorems~\ref{Thm:L2} and~\ref{Thm:H1}. We use the following nonlinear Poisson equation in different spatial dimensions:
\begin{equation}
-\nabla \cdot (a(u)\nabla u) = f(x), x\in\Omega, \hspace{2cm}
u(x) = g(x), x\in\partial\Omega,
\end{equation}
where the domain is the $d$-dimensional cube
$\Omega=[-1,1]^d$ and $a(u)=u^2-u$. We set $f(x)$ and $g(x)$ to be
\begin{equation*}
f(x)
=
\frac{1}{d}
\left[
-3\exp\left(-\frac{3}{d}\sum_{i=1}^{d}x_i\right)
+2\exp\left(-\frac{2}{d}\sum_{i=1}^{d}x_i\right)
\right],\quad g(x) = \exp\left(-\frac{1}{d}\sum_{i=1}^{d}x_i\right).
\end{equation*}
The analytical solution is given by $u(x)=\exp\left(-\frac{1}{d}\sum_{i=1}^{d}x_i\right)$. 

For $d=2$, we sample random weights from Gaussian distribution with variance $\sigma^2=0.1$, while for dimensions $d=4$ and $d=8$, we sample random weights from Gaussian distributions with variance $\sigma^2=0.05$.
For all dimensions, we sample $400$ interior points and $80$ boundary points as training points, and $2000$ interior points and $208$ boundary points as testing points. For each case, we perform $2000$ Adam iterations followed by $500$ L-BFGS iterations to train the model.

We show the log-log plots of $L^2$ and $H^1$ errors versus number of features $N$ for dimensions $d=2,4, \mbox{and }8$ in Figure~\ref{fig:Convergence_Rate}. The estimated decay rates are also included in the legends. 
We observe that the estimated decay rates are consistent with those established in Theorems~\ref{Thm:L2} and~\ref{Thm:H1}, demonstrating the potential of our proposed method for solving high-dimensional PDEs. 
\begin{figure}[ht]
\centering 
\includegraphics[width=0.38\linewidth]{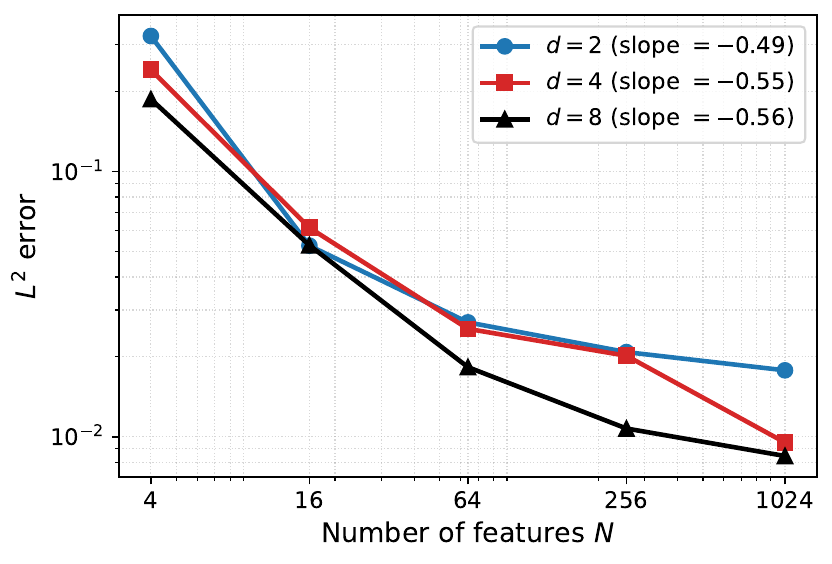}
\includegraphics[width=0.38\linewidth]{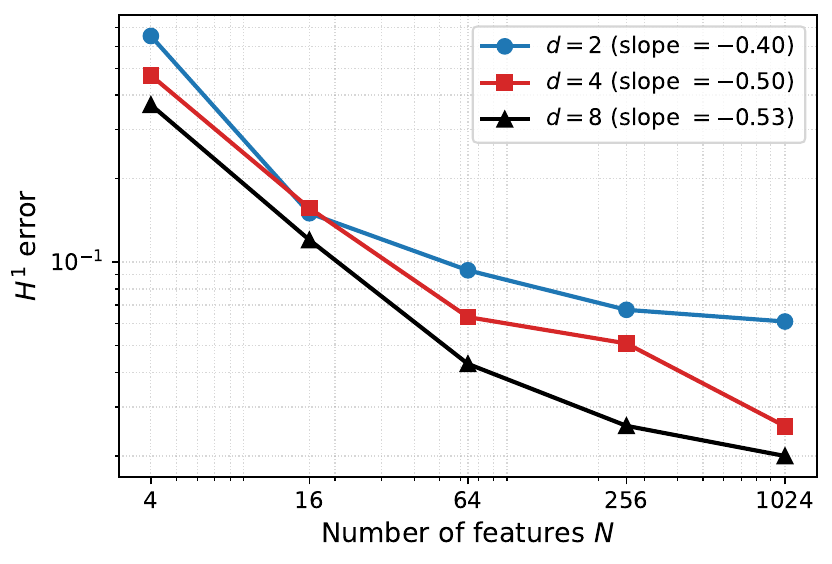}        
\caption{Log-log plots of $L^2$ error (left) and $H^1$ error (right) versus number of features $N$ for dimensions $d=2, 4, \mbox{and } 8$. The plots numerically verify the decay rates established in Theorems~\ref{Thm:L2} and~\ref{Thm:H1} }
\label{fig:Convergence_Rate}
\end{figure}

\subsection{Linear Advection-Diffusion Equation}\label{sec:LADE}
In this section, we compare our proposed product cosine random features with the uniform random feature proposed in \cite{Liao2026PDE}. We consider the following linear advection-diffusion equation 
\begin{equation}\label{eq:linear_advection}
\begin{cases}
u_t + cu_x = \nu u_{xx}, & (x,t) \in (-1,1) \times (0,1],\\
u(-1,t) = u(1,t), & t \in (0,1],\\
u(x,0) = \sin(\pi x), & x \in (-1,1),
\end{cases}
\end{equation}
whose exact solution is $
u(x,t) = \exp(-\nu \pi^2 t)\sin\bigl(\pi(x-ct)\bigr)$.
We set the wave speed to $c=5$ and the diffusion coefficient to $\nu=0.1$.

The uniform random feature model concatenate the spatial and temporal variables and therefore the random weights are sampled from the same distribution. 
In contrast, the product random feature representation treats the spatial and temporal variables separately. In this example, we demonstrate that the product feature representation can outperform the uniform random feature representation, especially when the spatial and temporal variables have different scales.

For each trial, we use the same set of training and testing points for both models. Specifically, we sample $5000$ interior points, $200$ boundary points ($100$ on each side), and $100$ initial points as training points, and $10000$ points as testing points.
For the comparison, we use the same total number of random features for both models. For the uniform random feature model, we set the number of features to $N=3000$ and the standard deviation parameter to $\sigma=10.0$. For the product random feature (PRF) model, we use $N_x=30$ spatial features and $N_t=100$ temporal features, giving a total of $N_xN_t=3000$ product features. The corresponding variance parameters are set to $\sigma_x=3.0$ and $\sigma_t=17.0$.

Table~\ref{tab:linear_advection_diffusion} compares the performance of the uniform and product random feature models. The PRF model achieves substantially lower RMSE and $L^{\infty}$ errors than the uniform random feature model. This result is expected, as the spatial and temporal variables have different scales ($c=5$ and $\nu=0.1$). 
The product random feature can effectively accommodate this difference by sampling the weights from different distributions. 
\begin{table}[ht]
\centering
\begin{tabular}{lccc}
\hline
Method & Relative $L^{2}$ Error & $L^{\infty}$ Error \\ 
\hline
Product RF & $ 1.6 \times 10^{-3} \pm 5.7 \times 10^{-4}$ & $ 7.6\times 10^{-3} \pm 4.3 \times 10^{-3}$ \\
Uniform RF & $1.6 \times 10^{-2} \pm 2.1 \times 10^{-2} $ & $2.5 \times 10^{-2} \pm 2.6\times 10^{-2} $ \\ 
\hline
\end{tabular}
\caption{Comparison of product and uniform random feature methods. We report average relative $L^2$ error and $L^\infty$ error accompanying with the standard deviations. Each experiment is averaged over 10 trials.} 
\label{tab:linear_advection_diffusion}
\end{table}

\subsection{Helmholtz Equation}
\label{sec:Helm}
In this section, we consider the 2D Helmholtz equation, which is a benchmark to test the ability of the method to approximate oscillatory solutions and to handle a second-order elliptic operator. The equation is given by
\begin{equation}
u_{xx} + u_{yy} + k^{2}u - q(x,y) = 0, (x,y)\in\Omega, \hspace{2cm} u(x,y)=0, (x,y)\in\partial\Omega,
\end{equation}
where the domain $\Omega=[-1,1]\times[-1,1]$ and the function $q(x,y)$ is chosen as 
$$q(x,y) = \big(k^{2} - (a_{1}\pi)^{2} - (a_{2}\pi)^{2}\big)   \sin(a_{1}\pi x)\sin(a_{2}\pi y).$$
The analytic solution is $u(x,y) = \sin(a_{1}\pi x) \sin(a_{2}\pi y)$. 
We set the parameters $k=1$, $a_{1}=1$, and $a_{2}=4$. In this example, we apply the product random feature because we observe the different scales of $x$ and $y$. 
We report the number of features $N_x$ and $N_t$ for spatial and temporal variables in Table~\ref{tab:model_settings_1}. Weights are randomly sampled from Gaussian distributions with standard deviations $\sigma_x$ and $\sigma_t$, respectively.
We compare with PINN, SA-PINN, and ELM. The networks architectures, activation functions,  and training settings are summarized in Table~\ref{tab:model_settings_1}.
For each trial, we use $4800$ interior points and 400 boundary points ($100$ on each side) for training, and $10{,}000$ points for testing. 
\begin{table}[htbp]
    \centering
    \begin{tabular}{lll}
        \hline
        Method & Model setting & Optimization \\
        \hline
        PIRFN
        & \makecell{$N_x=40,\ N_y=160,$ \\ $\sigma_x=2.0,\ \sigma_y=8.0$} 
        & Adam: $2500$; L-BFGS: $3000$ \\
        
        PINN
        & $\tanh$, $[2,50,50,50,50,1]$
        & Adam: $5000$; L-BFGS: $5000$ \\

        SA-PINN
        & $\tanh$, $[2,50,50,50,50,1]$
        & Adam: $10000$; L-BFGS: $500^{*}$ \\

        ELM
        & $\tanh$, $D=6400$
        & Adam: $2500$; L-BFGS: $3000$ \\
        \hline
    \end{tabular}
    \label{tab:model_settings_1}
    \caption{Model architectures and training settings for the 2D Helmholtz equation (numbers marked with * indicate early stopping).}
\end{table}

Figure~\ref{fig:helmholtz_solution} shows the analytic solution and the approximated solution obtained using the random feature method.  Table~\ref{tab:helmholtz_results} summarizes relative $L^2$ error and $L^\infty$ error for each method over 10 independent trials.  Our proposed PIRFN method achieves substantially lower errors than the other comparison methods in both errors.  These results indicate that our method can accurately capture the oscillatory solution of the two-dimensional Helmholtz equation.
\begin{figure}[ht]
\centering
\includegraphics[width=0.75\linewidth]{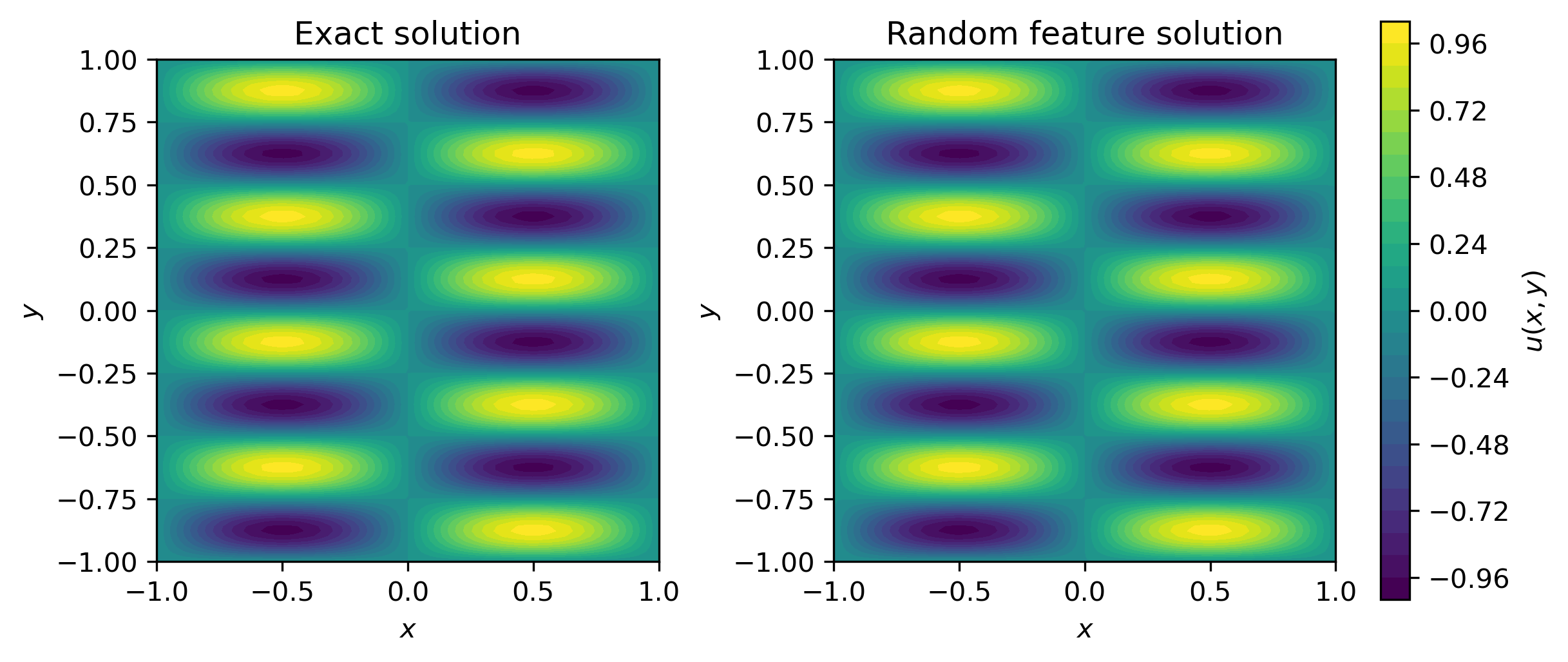}
\caption{Exact solution (left) and random feature approximation (right) for the 2D Helmholtz equation.}
\label{fig:helmholtz_solution}
\end{figure}

\begin{table}[ht]
\centering
\begin{tabular}{lcc}
\hline
Method & Relative $L^{2}$ Error & $L^{\infty}$ Error \\
\hline
PIRFN 
& $1.5 \times 10^{-2} \pm 1.5 \times 10^{-2}$ 
& $2.8 \times 10^{-2} \pm 1.5 \times 10^{-2}$ \\

PINN 
& $4.8 \times 10^{-1}\pm 7.7 \times 10^{-1}$ 
& $1.1 \pm 1.8$ \\

SA-PINN 
& $1.76 \pm 7.4 \times 10^{-1}$ 
& $3.00 \pm 1.01$ \\ 

ELM 
& $6.56 \pm 5.11$ 
& $8.63 \pm 5.44$ \\
\hline
\end{tabular}
\caption{2D Helmholtz Equation: We report relative $L^2$ error and $L^\infty$ error for each method. Errors are averaged over
$10$ independent trials.}
\label{tab:helmholtz_results}
\end{table}
\subsection{Linear Transport}
In this section, we consider the following periodic linear transport equation 
\begin{equation}
\begin{cases}
    u_t + c u_x = 0, & (x,t) \in (-1,1)\times(0,1),\\
    u(0,x) = \sin(\pi x), & x\in(-1,1),\\
    u(t,-1) = u(t,1), & t\in(0,1),
\end{cases}
\end{equation}
whose analytic solution is $u(t,x) = \sin\bigl(\pi(x-ct)\bigr)$.
Following the challenging transport regime considered in the PINN literature \cite{krishnapriyan2021characterizing}, we set the wave speed to $c=12$. Although this value is not large in an absolute sense, wave speeds beyond $c=10$ have been reported to cause significant training difficulties for vanilla PINNs. Thus, this example provides a challenging test of the ability of the learning methods to accurately capture rapidly propagating solutions.  The model architectures, activation functions, and training settings for this example are summarized in Table~\ref{tab:model_settings_2}. For each trial, we use $5000$ interior points, $200$ boundary points ($100$ on each side), and $100$ initial points for training, and $10{,}000$ points for testing.
\begin{table}[htbp]
    \centering
    \begin{tabular}{lll}
        \hline
        Method & Model setting & Optimization \\
        \hline
      PIRFN
        & \makecell{$N_x=30,\ N_t=120$,\\ $\sigma_x=3.0,\ \sigma_t=36.0$}         
        & Adam: $2500$; L-BFGS: $6000$ \\
        
        PINN
        & $\tanh$, $[2, 20 \times 10, 1]$
        & Adam: $6000$; L-BFGS: $500^{*}$ \\

        SA-PINN
        & $\tanh$, $[2, 20 \times 10,1]$
        & Adam: $6000$; L-BFGS: $500^{*}$ \\

        ELM
        & $\tanh$, $D=3600$
        & Adam: $2500$; L-BFGS: $1500$ \\
        \hline
    \end{tabular}
    \caption{Model architectures and training settings for the
    linear transport equation (numbers marked with * indicate early stopping).}
    \label{tab:model_settings_2}
\end{table}
In Figure~\ref{fig:linear_transport}, we depict the true solution and the random feature. We observe that the random feature approximation accurately captures the propagation of the wave throughout the space-time domain. Furthermore, as shown in Table~\ref{tab:linear_transport_results}, random feature method achieves lower errors compared with PINN and its variant, highlighting the difficulty of this problem for PINN-based approaches and the potential of random feature-based method for challenging space-time PDEs.
\begin{figure}[ht]
\centering
\includegraphics[width=\linewidth]{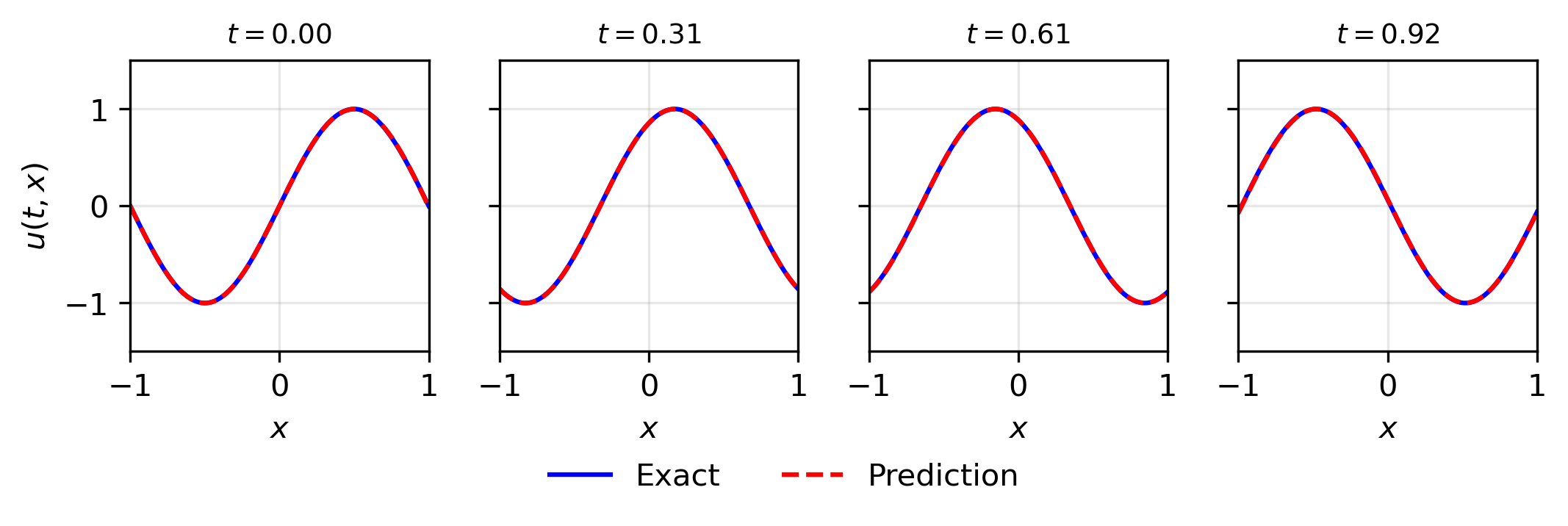}
\caption{Exact solution and random feature approximation for the linear transport equation with $c=12$.}
\label{fig:linear_transport}
\end{figure} 
\begin{table}[ht]
\centering
\begin{tabular}{lcc}
\hline
Method & Relative $L^{2}$ Error & $L^{\infty}$ Error \\
\hline
PIRFN   & $5.30 \times 10^{-3} \pm 2.97 \times 10^{-3}$ 
        & $2.64 \times 10^{-2} \pm 2.20 \times 10^{-2}$ \\
PINN   & $1.00 \pm 1.46 \times 10^{-1}$ & $1.37 \pm 2.68 \times 10^{-1}$ \\
SA-PINN & $9.60 \times 10^{-1} \pm 4.88 \times 10^{-3}$ 
        & $1.23 \pm 4.33 \times 10^{-1}$ \\
ELM    & $1.16 \pm 6.93 \times 10^{-2}$ 
        & $1.51 \pm 8.84 \times 10^{-2}$ \\
\hline
\end{tabular}
\caption{Linear Transport Equation: We report relative $L^2$ error and $L^\infty$ error for each method. Errors are averaged over 10 independent trials.}
\label{tab:linear_transport_results}
\end{table}
\subsection{Wave Equation}\label{sec:wave}
In this section, we consider the following 1D wave equation
\begin{equation}
\begin{cases}
u_{tt} - 4u_{xx} = 0, & (x,t)\in(0,1)\times(0,1),\\
u(x,0) = \sin(\pi x) + \frac{1}{2}\sin(4\pi x), & x\in[0,1],\\
u(0,t)=0,\quad u(1,t)=0, & t\in(0,1),\\
u_t(x,0)=0, & x\in[0,1].
\end{cases}
\end{equation}
The analytical solution is given by
\begin{equation*}
u(x,t) = \sin(\pi x)\cos(2\pi t) + \frac{1}{2}\sin(4\pi x)\cos(8\pi t).
\end{equation*}
This example involves multiple spatial and temporal frequencies and therefore provides a test of the ability of the learning methods to accurately capture multi-frequency spatiotemporal dynamics.
We compare our proposed method with PINN, SA-PINN, and ELM and summarize the model architectures and training settings for each method in Table~\ref{tab:model_settings_3}. 
For the collocation point used to train the models, we uniformly sample $5000$ interior points, $200$ boundary points ($100$ on each side), and $100$ initial points. We independently sample another set of $10{,}000$ points uniformly from the domain for testing.
\begin{table}[htbp]
    \centering
    \begin{tabular}{lll}
        \hline
        Method & Model setting & Optimization \\
        \hline
        PIRFN
        & \makecell{$N_x=70,\ N_t=140,$\\ $\sigma_x=8.0,\ \sigma_t=16.0$}        
        & Adam: $2500$; L-BFGS: $8000$ \\

        PINN
        & $\tanh$, $[2, \ 30 \times 12, \ 1]$
        & Adam: $3000$; L-BFGS: $1500$ \\

        SA-PINN
        & $\tanh$, $[2, \ 30 \times 12, \ 1]$
        & Adam: $5000$; L-BFGS: $500^{*}$ \\
        
        ELM
        & $\tanh$, $D=9800$
        & Adam: $2500$; L-BFGS: $1500$ \\
        \hline
    \end{tabular}
    \caption{Model architectures and training settings for the 1D
    wave equation (numbers marked with * indicate early stopping).}
    \label{tab:model_settings_3}
\end{table}

We first compare the random feature approximation with the exact solution in Figure~\ref{fig:wave}, which shows that the random feature  approximation has the ability to capture the spatiotemporal evolution of the wave and reproduces the solution accurately throughout the space-time domain. 

\begin{figure}[ht]
\centering
\includegraphics[width=\linewidth]{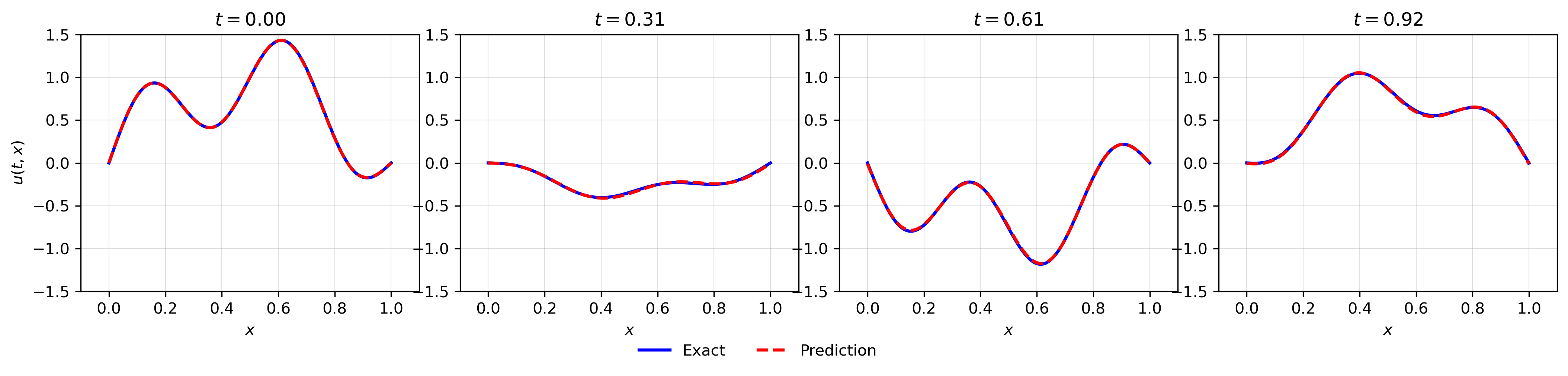}
\caption{Comparison of the exact solution and the random feature approximation for the 1D wave equation.}    
\label{fig:wave}
\end{figure} 
In addition, we report relative $L^2$ error and $L^\infty$ error of each method in Table~\ref{tab:wave_results}. 
As shown in the table, PIRFN achieves an $L^2$ relative error on the order of $10^{-2}$, whereas all other methods yield errors on the order of $10^{-1}$, corresponding to an improvement of approximately one order of magnitude.
Similarly, PIRFN achieves the smallest $L^\infty$ error among all other methods, demonstrating its effectiveness in this multi-frequency oscillatory setting.

\begin{table}[!ht]
\centering
\begin{tabular}{lccc}
\hline
Method & Relative $L^{2}$ Error & $L^{\infty}$ Error \\
\hline
PIRFN & $ 4.29  \times 10^{-2} \pm 4.31 \times 10^{-2} $ & $ 1.04 \times 10^{-1} \pm 8.91 \times 10^{-2} $ \\ 
PINN & $9.51 \times 10^{-1} \pm 6.51 \times 10^{-2}$ & $1.43 \pm 6.80 \times 10^{-2}$  \\
SA-PINN & $6.23 \times 10^{-1} \pm 1.92 \times 10^{-1} $ &$9.50 \times 10^{-1} \pm 3.23 \times 10^{-1} $\\
ELM & $ 4.57 \times 10^{-1} \pm 5.55 \times 10^{-3}$ & $6.08 \times 10^{-1} \pm 2.52 \times 10^{-2} $\\
\hline
\end{tabular}
\caption{1D Wave Equation: We report relative $L^2$ error and $L^\infty$ error for each method. Errors are averaged over 10 independent trials.}
\label{tab:wave_results}
\end{table}
\section{Conclusions}\label{sec:conclude}
In this work, we developed a physics-informed random feature network framework for solving partial differential equations. The framework combines kernel-induced random features with a physics-informed loss, optimizing only the output layer coefficients through automatic differentiation and Adam combined with L-BFGS training. We introduced a product random feature representation that allows separate choices of sampling distributions and feature counts for spatial and temporal variables, or for different spatial directions. This construction provides flexibility in representing solutions with distinct characteristic scales. On the theoretical side, we derived high-probability approximation error bounds in the Sobolev $H^1$ norm under suitable assumptions on the target function and random features.

Numerical experiments demonstrated the approximation capabilities of the proposed framework across several PDE problems. Numerical studies for a nonlinear Poisson equation in dimensions two, four, and eight exhibited error decay consistent with the theoretical estimates. For the advection–diffusion problem, the product representation improved accuracy over the uniform random feature model at the same total feature count. On the Helmholtz, linear transport, and wave equation benchmarks, the proposed method achieved lower errors than the tested PINN, self-adaptive PINN, and ELM configurations. These findings highlight the importance of matching random feature distributions to the characteristic scales of the solution and support further investigation of this approach for oscillatory and multiscale PDEs.

Several directions remain for future work. Extending the approximation analysis to account for finite collocation sampling, optimization error, and PDE stability would provide a more complete understanding of the accuracy of the trained solver.  The effect of using the minimal RKHS norm in the optimization setting~\eqref{MinnormRF_PDE} is another important discussion on the well-posedness of the PIRFN's optimization problem.  Developing automatic or adaptive strategies for selecting kernels and sampling distributions is another important step, as these choices currently depend on the problem. Finally, localized or adaptive random feature representations may improve the treatment of localized non-smooth structures, including shocks, and broaden the applicability of the framework.
\appendix
\section{Some Useful Concentration Inequalities}
In this section, we recall some classical concentration inequalities from \cite{Pinelis1986remark} that estimates the difference between empirical averages and true averages of random vectors. 

\begin{theorem}[Vector-valued Bernstein inequality in Hilbert space]
\label{Thm:Bern}
Let $Z$ be a H-valued random variable, where $(H,\langle\cdot,\cdot\rangle,\|\cdot\|)$ is a separable Hilbert space. Suppose there exist positive numbers $b>0$ and $\sigma>0$ such that
\begin{equation}
\label{B_mom_cond}
\Eb \| Z - \Eb Z\|^p \leq \frac{1}{2}p!\sigma^2b^{p-2} \quad \mbox{ for all } p\geq 2.
\end{equation}
For any $\delta\in(0,1)$ and $N\in\N$, denoting by $\{Z_n\}_{n=1}^N$ a sequence of $N$ i.i.d copies of $Z$. it holds that
\begin{equation}
\Pbb\left( \left\| \frac{1}{N}\sum_{n=1}^N Z_n - \Eb Z\right\| \leq \frac{2b\log(2/\delta)}{N} + \sqrt{\frac{2\sigma^2\log(2/\delta)}{N}}\right) \geq 1-\delta. 
\end{equation}
\end{theorem}
One of the most frequently used forms of Berstein's inequality is the following setting for bounded random variables. 

\begin{lemma}
\label{lemma:bound}
Let $Z$ be a (potentially) uncentered random variable such that
\begin{equation*}
\|Z\| \leq c \mbox{ almost surely } \quad \mbox{ and } \Eb\|Z-\Eb Z\|^2 \leq v^2 
\end{equation*}
for some $c>0$ and $v>0$. Then $Z$ satisfies Bernstein's moment condition \eqref{B_mom_cond} with $b=2c$ and $\sigma=v$. If $\Eb Z=0$, then taking $b=c$ suffices.
\end{lemma}

\section*{Acknowledgments}
Zhong is supported by NSF-CCF-AoF grant $\#2225507$. CL is supported by the start-up funding from University of Arkansas.
\bibliographystyle{siamplain}
\bibliography{RFN_bib}
\end{document}

%% file: marco.tex
\theoremstyle{plain}
\newtheorem{remark}[theorem]{Remark}

\newcommand{\yb}{\mathbf{y}}

\newcommand{\xb}{\mathbf{x}}
\newcommand{\cb}{\mathbf{c}}

\newcommand{\Ab}{\mathbf{A}}

\newcommand{\Ib}{\mathbf{I}}
\newcommand{\Kb}{\mathbf{K}}
\newcommand{\Eb}{\mathbf{E}}

\newcommand{\C}{\mathbb{C}}
\newcommand{\R}{\mathbb{R}}

\newcommand{\N}{\mathbb{N}}

\newcommand{\Pbb}{\mathbb{P}}

\newcommand{\Fc}{\mathcal{F}}
\newcommand{\Hc}{\mathcal{H}}

\newcommand{\Nc}{\mathcal{N}}
\newcommand{\Mc}{\mathcal{M}}
\newcommand{\Oc}{\mathcal{O}}
\newcommand{\Pc}{\mathcal{P}}
\newcommand{\Bc}{\mathcal{B}}
\newcommand{\Lc}{\mathcal{L}}

\newcommand{\irm}{\mathrm{i}}

\newcommand{\minimize}{\mathop{\mathrm{minimize}}}
\newcommand{\argmin}{\mathop{\mathrm{argmin}}}

\newcommand{\mix}{\mathrm{mix}}

\newcommand{\indicator}{\mathbbm{1}}
\newcommand{\omegab}{\boldsymbol\omega}